\documentclass[11pt]{article}
\usepackage[utf8]{inputenc}
\usepackage{amsfonts}
\usepackage{enumitem}
\usepackage[margin=3cm]{geometry}
\usepackage{amsmath,color}

\usepackage{float}
\usepackage{amssymb,amsthm,mathrsfs}
\usepackage{graphicx,cite}
\usepackage{bbm,hyperref}
\usepackage{spalign}

\hypersetup{
	colorlinks=true,
    linkcolor={red!50!black},
    citecolor={blue!50!black},
    urlcolor={blue!80!black},
	bookmarksopen=true,
	bookmarksnumbered,
	bookmarksopenlevel=2,
	bookmarksdepth=3
}
\usepackage{boxedminipage}
\usepackage{tabularx,lipsum,environ}
\usepackage{todonotes}
\usepackage[algo2e, vlined]{algorithm2e}
\usepackage{booktabs}
\usepackage{tabularx}
\usetikzlibrary{shapes}
\usepackage{url}

\usepackage{latexsym}
\usepackage{graphics}
\usepackage{amsbsy}
\usepackage{ifpdf}
\usepackage{amstext}

\usepackage[capitalise]{cleveref}

\newtheorem{theorem}{Theorem}[section]
\newtheorem{lemma}[theorem]{Lemma}

\newtheorem{corollary}[theorem]{Corollary}

\theoremstyle{definition}
\newtheorem{remark}[theorem]{Remark}
\newtheorem{example}[theorem]{Example}

\usepackage[mathlines]{lineno}
\usepackage{etoolbox}

\newcommand{\FPT}{$\mathsf{FPT}$~}

\newcommand{\NPh}{\hbox{$\mathsf{NP}$-hard~}}
\newcommand{\coNPc}{$\mathsf{co}$-$\mathsf{NP}$-complete~}
\newcommand{\W}{{\cal S}}

\newif\ifDissertation
\Dissertationtrue

\newcommand{\problemdef}[3]{
	\begin{center}
		\begin{boxedminipage}{.99\textwidth}
			\textsc{{#1}}\\[2pt]
			\begin{tabular}{ r p{0.8\textwidth}}
				\textbf{~~~~Instance:} & {#2}\\
				\textbf{Question:} & {#3}
			\end{tabular}
		\end{boxedminipage}
	\end{center}
}

\newcommand\WCW[1]{{\rm WCW}(#1)}

\newcommand\wcdim[1]{{\rm wcdim}(#1)}

\newcommand{\C}{\mathcal{C}}
\newcommand{\G}{\mathcal{G}}

\newcommand*\linenomathpatch[1]{%
  \cspreto{#1}{\linenomath}%
  \cspreto{#1*}{\linenomath}%
  \csappto{end#1}{\endlinenomath}%
  \csappto{end#1*}{\endlinenomath}%
}
\newcommand*\linenomathpatchAMS[1]{%
  \cspreto{#1}{\linenomathAMS}%
  \cspreto{#1*}{\linenomathAMS}%
  \csappto{end#1}{\endlinenomath}%
  \csappto{end#1*}{\endlinenomath}%
}

\expandafter\ifx\linenomath\linenomathWithnumbers
  \let\linenomathAMS\linenomathWithnumbers
  \patchcmd\linenomathAMS{\advance\postdisplaypenalty\linenopenalty}{}{}{}
\else
  \let\linenomathAMS\linenomathNonumbers
\fi

\linenomathpatch{equation}
\linenomathpatchAMS{gather}
\linenomathpatchAMS{multline}
\linenomathpatchAMS{align}
\linenomathpatchAMS{alignat}
\linenomathpatchAMS{flalign}
\usepackage[T1]{fontenc}

\title{Computing Well-Covered Vector Spaces of Graphs\\ Using Modular Decomposition}

\author{Martin Milani{\v c}\\
\small FAMNIT and IAM, University of Primorska\\
\small \texttt{martin.milanic@upr.si}\\
\and
Nevena Piva\v c\\
\small FAMNIT and IAM, University of Primorska\\
\small \texttt{nevena.pivac@iam.upr.si}
}

\begin{document}
\maketitle

\begin{abstract}
A graph is well-covered if all its maximal independent sets have the same cardinality.
This concept was introduced by Plummer in 1970 and naturally generalizes to the weighted case.
Given a graph $G$, a real-valued vertex weight function $w$ is said to be a well-covered weighting of $G$ if all its maximal independent sets are of the same weight with respect to~$w$.
The set of all well-covered weightings of a graph $G$ forms a vector space over the field of real numbers, called the \emph{well-covered vector space} of $G$.
Since the problem of recognizing well-covered graphs is $\mathsf{co}$-$\mathsf{NP}$-complete, the problem of computing the well-covered vector space of a given graph is $\mathsf{co}$-$\mathsf{NP}$-hard.
Levit and Tankus showed in 2015 that the problem admits a polynomial-time algorithm in the class of claw-free graphs.
In this paper, we give two general reductions for the problem, one based on anti-neighborhoods and one based on modular decomposition, combined with Gaussian elimination.
Building on these results, we develop a polynomial-time algorithm for computing the well-covered vector space of a given fork-free graph, generalizing the result of Levit and Tankus.
Our approach implies a polynomial-time recognition algorithm for the class of well-covered fork-free graphs and also generalizes some known results on cographs.


\medskip
\noindent{\bf MSC Codes (2020):}
05C85 (Primary) 05C69, 05C75, 05C50, 15A03 (Secondary)
\end{abstract}

\section{Introduction}

An \emph{independent set} in a graph $G$ is a set of pairwise nonadjacent vertices.
An independent set in a graph $G$ is said to be \textit{maximum} if it has maximum cardinality and \textit{maximal} if it is not contained in any larger independent set.
The problem of finding a maximum independent set in a given graph, known as \textsc{Maximum Independent Set}, is one of the classical \NPh problems~\cite{MR0378476}.
While every maximum independent set in a graph is also a maximal one, the opposite implication does not hold in general.
If every maximal independent set in a graph $G$ is also a maximum one, the graph $G$ is said to be \textit{well-covered}.
Well-covered graphs were introduced by Plummer in 1970~\cite{zbMATH03310754} and have been extensively studied in the literature (see~\cite{MR1677797} for an introduction and~\cite{MR1254158} for a survey).
One of the motivations for the study of well-covered graphs stems from the fact that \textsc{Maximum Independent Set} is solvable in linear time in the class of well-covered graphs by a simple greedy algorithm that computes a maximal independent set.

Two central research directions on well-covered graphs are the study of their recognition and their characterizations in special graph classes.
As proved independently by Sankaranarayana and Stewart in 1992~\cite{zbMATH00039777} and by Chv\'atal and Slater in 1993~\cite{zbMATH00434906}, the recognition of well-covered graphs is $\mathsf{co}$-$\mathsf{NP}$-complete.
In Plummer's survey from 1993 (see~\cite{MR1254158}) one can find results on various restrictions of the well-coveredeness property defining special subclasses of well-covered graphs, as well as an overview of the study of well-coveredeness versus the girth and the maximum degree.
After Plummer's survey, the study of well-covered graphs focused mostly on the recognition problem in special cases.
In particular, Caro, Seb\H{o}, and Tarsi showed that the recognition of well-covered graphs remains \coNPc even for $K_{1,4}$-free graphs~\cite{zbMATH00845871}, Brown and Hoshino established $\mathsf{co}$-$\mathsf{NP}$-completeness for the class of circulant graphs~\cite{MR2739910}, and a careful examination of the reduction due to Sankaranarayana and Stewart~\cite{zbMATH00039777} shows that the problem remains \coNPc in the class of weakly chordal graphs, that is, graphs such that neither the graph nor its complement contain an induced cycle of length at least five.
On the positive side, Tankus and Tarsi showed that the problem is polynomial-time solvable in the class of claw-free graphs (see~\cite{zbMATH00934795, zbMATH01006762}).
The well-coveredness property can also be tested efficiently in the classes of bipartite graphs~\cite{zbMATH03614795,zbMATH03799694}, graphs with girth at least $5$ \cite{zbMATH00130578}, graphs without cycles of lengths $4$ and $5$~\cite{zbMATH00682532}, chordal graphs~\cite{zbMATH00845475}, graphs of bounded degree~\cite{zbMATH01200865}, perfect graphs with bounded clique number~\cite{MR1264476}, various generalizations of the class of cographs~\cite{Klein2013,Araujo2019}, and graphs of bounded cliquewidth~\cite{MR3851408}.
The problem has also been studied from the parameterized complexity point of view, by Alves, Dabrowski, Faria, Klein, Sau, and Souza~\cite{MR3851408} and Ara\'{u}jo, Costa, Klein, Sampaio, and Souza~\cite{Araujo2019}.

In this paper we focus on a weighted generalization of well-coveredness.
If every vertex of a graph $G$ is assigned a real number, that is, the \emph{weight} of a vertex, we speak about a \emph{weighted graph}.
\textsc{Maximum Weight Independent Set} is the problem of computing an independent set of maximum weight in a given weighted graph, where the weight of a set of vertices is defined as the sum of the weights of its members.
Given a graph $G$ and a weight function $w:V(G)\to \mathbb{R}$, a graph $G$ is said to be \emph{\hbox{$w$-well-covered}} if all maximal independent sets in $G$ are of the same weight with respect to the weight function $w$.
The concept of $w$-well-covered graphs was introduced by Caro, Ellingham, and Ramey in 1998~\cite{zbMATH01200865}, in the more general context of weight functions mapping the vertices of a graph to the elements of an abelian group (see also~\cite{zbMATH05029664}).
Graphs that are \hbox{$w$-well-covered} with respect to some \emph{nonnegative} weight function $w:V(G)\to \mathbb{R}_+$ that is not identically equal to $0$ are exactly the complements of the so-called \emph{stochastic} graphs studied already in 1983 by Berge~\cite{zbMATH03813669}, and generalize the \emph{equistable} graphs introduced in 1980 by Payan~\cite{MR553649} and defined as the graphs that admit a weight function $w:V(G)\to \mathbb{R}_+$ such that a set $S\subseteq V(G)$ is a maximal independent set if and only if the total weight of $S$ equals $1$.

Given a graph $G$, a \emph{well-covered weighting} of $G$ is any real-valued weight function $w$ on the vertices of $G$ such that $G$ is $w$-well-covered.
For every graph $G$, the set $\WCW{G}$ of all well-covered weightings of $G$ forms a vector space over the field of real numbers (see~\cite{zbMATH01334515,zbMATH05029664}); we refer to it as the \emph{well-covered vector space} of $G$.
Similar vector spaces can be defined for more general situations, for example for hypergraphs and for vertex weight functions that assign to each vertex of $G$ a value from some fixed field $\mathbb{F}$ (see Caro and Yuster~\cite{zbMATH01334515} and Brown and Nowakowski~\cite{zbMATH05029664}).
In this paper we restrict our attention to the case of graphs and the field of real numbers.
Any system of equations representing the vector space $\WCW{G}$ will be referred to as a \emph{well-covering system} of $G$.
(Precise definitions will be given in~\Cref{sec:prelim}.)
We consider the following problem.

\problemdef{{Well-Covering System}}{A graph $G=(V,E)$.}{Compute a well-covering system of $G$.}

A graph is well-covered if and only if the vertex weight function that is constantly equal to $1$ belongs to the \hbox{well-covered} vector space of the graph.
Therefore, since the problem of recognizing well-covered graphs is $\mathsf{co}$-$\mathsf{NP}$-complete, the more general \textsc{Well-Covering System} problem is  $\mathsf{co}$-$\mathsf{NP}$-hard.

The \emph{well-covered dimension} of $G$ is denoted by $\wcdim{G}$ and defined as the dimension of the well-covered vector space of $G$.
Clearly, a graph $G$ has well-covered dimension equal to zero if and only if the only well-covered weighting of $G$ is the identically zero function.
Such graphs are known to exist; for instance, the Petersen graph and any cycle with at least $8$ vertices are among them (see~\cite{zbMATH01200865,zbMATH01334515}).
However, to the best of our knowledge, the complexity of computing the well-covered dimension of a graph is open, even in the special case of recognizing graphs with positive well-covered dimension.
Caro and Yuster proved that the well-covered dimension of a tree is equal to the number of leaves~\cite{zbMATH01334515}.
Brown and Nowakowski generalized this result to the class of chordal graphs~\cite{zbMATH05029664} by showing that in this case the well-covered dimension equals the number of simplicial cliques.
They also showed that the well-covered dimension can be computed in polynomial time for cographs, for graphs with independence number at most two, and for chordal graphs.
The well-covered dimension of certain product graphs was studied by Birnbaum, Kuneli, McDonald, Urabe, and Vega~\cite{MR3268692} and for Levi graphs of point-line configurations by Hauschild, Ortis, and Vega~\cite{MR3404664}.

Well-covered vector spaces of graphs containing no cycles of length $4$ were studied by Brown, Nowakowski, and Zverovich~\cite{brown2007structure}.
\textsc{Well-Covering System} can be solved in polynomial time in classes of graphs of bounded vertex degree, as shown by Caro, Ellingham, and Ramey~\cite{zbMATH01200865}, in the class of graphs with girth at least $7$, as shown by Caro and Yuster~\cite{zbMATH01334515}, and, as shown by Levit and Tankus, in the class of claw-free graphs~\cite{levit2015weighted} and in the class of graphs without cycles of lengths $4$, $5$, and $6$~\cite{MR3325553}.

\subsection*{Our results and relation with existing works}

In this paper, we give two general reductions for the \textsc{Well-Covering System} problem, one based on modular decomposition and one based on anti-neighborhoods.
Building on these results, we develop a polynomial-time algorithm for solving the problem in the class of fork-free graphs, thereby generalizing the analogous result of Levit and Tankus on claw-free graphs~\cite{levit2015weighted}.
The algorithm decomposes a given fork-free graph $G$ into a polynomial number of induced subgraphs of $G$ and recursively computes a well-covering system for every graph $H$ constructed at some step of the decomposition of $G$.
To keep the well-covering systems polynomially bounded in size, Gaussian elimination is applied at each step.
In the base case, when the subgraph $H$ cannot be decomposed further, we use a structural result on fork-free graphs due to Lozin and Milani\v{c}~\cite{MR2373852,lozin2008polynomial}  (see also~\cite{MR4276974}) to infer that $H$ is claw-free; hence, in this case the algorithm of Levit and Tankus applies.

The class of fork-free graphs generalizes the class of cographs, hence our results generalize the result of Brown and Nowakowski~\cite{zbMATH05029664} that the well-covered dimension of cographs can be computed in polynomial time.
Furthermore, our reduction involving modular decomposition generalizes the analogous reduction for the (unweighted) well-covered graphs due to Klein, de Mello, and Morgana~\cite{Klein2013}, who used modular and primeval decompositions to develop efficient algorithms for the problem of recognizing well-covered graphs in several extensions of the class of cographs.

\begin{sloppypar}
Our main result, a polynomial-time algorithm for computing a well-covering system of a given fork-free graph, is another example of an application of the structural result from~\cite{MR2373852,lozin2008polynomial} relating fork-free graphs to claw-free graphs (the exact statement of the result is given in~\Cref{thm:fork-free}).
This result was first used for developing a polynomial-time algorithm for \textsc{Maximum Weight Independent Set}  in the class of fork-free graphs~\cite{MR2373852,lozin2008polynomial}.
More recently it has been used by Dyer, Martin, Jerrum, M\"{u}ller, and Vu{\v{s}}kovi\'{c}~\cite{MR4276974} for developing a fully polynomial randomized approximation scheme (FPRAS)  for the problem of counting all weighted independent sets in a (fork, odd hole)-free graph~\cite{MR4276974} and by D{\k{e}}bski, Lonc, Okrasa, Piecyk, and Rz{\k{a}}{\.z}ewski~\cite{debski2022computing} for developing a polynomial-time algorithm for \textsc{$W_5$-Coloring} (a certain homomorphism problem) in the class of fork-free graphs.
\end{sloppypar}

\subsection*{Structure of the paper}

In~\Cref{sec:prelim} we collect the necessary definitions and preliminaries, including preliminaries on modular decomposition and precise definitions about well-covering systems.
In~\Cref{sec:prime} we show how to use modular decomposition, combined with Gaussian elimination, to reduce the problem of computing a well-covering system of a graph to the same problem on certain prime induced subgraphs of the graph.
In~\Cref{sec:cographs} we develop an algorithm for computing a well-covering system of a given cograph that is faster than the algorithm following from the main result of~\Cref{sec:prime}.
In~\Cref{sec:anti-neighborhoods} we show how to reduce the problem of computing a well-covering system of a graph to the same problem on the subgraphs of a given graph obtained by deleting the closed neighborhood of some vertex.
In~\Cref{sec:fork-free} we apply the results from~\Cref{sec:prime,sec:anti-neighborhoods} to develop a polynomial-time algorithm for computing a well-covering system of a given fork-free graph.
We conclude the paper with a summary and some open problems in~\Cref{sec:conclusion}.

\section{Preliminaries}\label{sec:prelim}

Given a positive integer $n$, we denote by $[n]$ the set $\{1,\ldots, n\}$ (and  $[0] := \emptyset$).
All graphs in this paper are finite, simple, and undirected.
Given a graph $G$, we denote by $V(G)$ the vertex set of $G$ and by $E(G)$ the edge set of $G$.
Two vertices $u$ and $v$ are \textit{adjacent} in $G$ if $uv\in E(G)$.
The \textit{neighborhood} of a vertex $v$ in $G$ is the set of all vertices adjacent to $v$ in $G$ and it is denoted by $N_G(v)$.
The \textit{closed neighborhood} of $v$ is defined as $N_G[v]=N_G(v)\cup \{v\}$.
Given a subset $S$ of vertices in $G$, we denote by $G[S]$ the graph induced by $S$, that is, the subgraph of $G$ with vertex set $S$ and edge set $\{uv: u,v\in S, uv\in E(G)\}$.
A \emph{connected component} of $G$ is a maximal connected subgraph of $G$.
The \emph{complement} of a graph $G = (V,E)$ is a graph $\overline{G}$ with vertex set $V$ in which two distinct vertices are adjacent if and only if they are nonadjacent in $G$.
A \emph{co-connected component}, or simply \emph{co-component}, of $G$ is the subgraph of $G$ induced by the vertex set of a connected component of the complement of $G$.
A graph is \emph{co-connected} if its complement is connected.
Given two disjoint subsets $A$ and $B$ of $V(G)$, we say that $A$ and $B$ are \emph{complete} (resp., \emph{anticomplete}) to each other in $G$ if
$\{ab : a \in A, b \in B\} \subseteq E(G)$ (resp., $\{ab : a \in A, b \in B\} \cap E(G) = \emptyset)$.

For integers $m,n\ge 0$ we denote by $K_{m,n}$ the \emph{complete bipartite graph} with parts of sizes $m$ and $n$, that is, the graph whose vertices can be partitioned into two independent sets $A$ and $B$, such that $|A|=m$, $|B|=n$, and $A$ and $B$ are complete to each other.
A \emph{claw} is the graph $K_{1,3}$.
A \emph{fork} is the graph obtained from a claw by a single subdivision of one of its edges, that is, the graph with vertex set $\{v_1,v_2,v_3,v_3,v_5\}$ and edge set $\{v_1v_2,v_2v_3,v_3v_4,v_3v_5\}$.
By $P_n$ we denote the $n$-vertex \emph{path graph}, that is, a graph whose vertices can be linearly ordered so that two vertices are adjacent if and only if they appear consecutively in the ordering.
Given two graphs $G$ and $H$, the graph $G$ is said to be \emph{$H$-free} if it contains no induced subgraph isomorphic to $H$.

A \emph{rooted tree} is a pair $(T,r)$ where $T$ is a tree and $r\in V(T)$ is the \emph{root} of $T$.
Given two nodes $u$ and $v$ in a rooted tree $T$, we say that $v$ is a \emph{child} (or \emph{successor}) of $u$ if $uv\in E(T)$ and $u$ belongs to the unique $v,r$-path in $T$.
A \emph{leaf} of a rooted tree $T$ is a node without any successors, while an \emph{internal node} of $T$ is a node that is not a leaf.
Note that if $T$ is a one-vertex rooted tree, then the unique vertex in $T$ is both the root and a leaf of $T$, but it is not an internal node.
Given a rooted tree $T$, we denote by $\ell(T)$ the number of leaves of $T$ and by $i(T) = |V(T)|-\ell(T)$ the number of internal nodes of $T$.
We will need the following well-known property of rooted trees.
To keep the paper self-contained, we include a proof.

\begin{lemma}\label{leaves in trees}
Let $T$ be a tree in which each internal node has at least two successors.
Then $\ell(T)\ge i(T)+1$.
\end{lemma}

\begin{proof}
By induction on $i(T)$.
If $i(T) = 0$, then the unique vertex in $T$ is a leaf and the inequality holds.
Let now $T$ be a tree with $i(T)\ge 1$ such that each internal node of $T$ has at least two successors, and assume that every tree $T'$ with $i(T')<i(T)$ in which each internal node has at least two successors satisfies that $\ell(T')\ge i(T')+1$.
Let $r$ be the root of $T$ and let $d$ be the number of successors of $r$.
Since $i(T)\ge 1$, the root of $T$ is an internal node.
Hence $d\ge 2$.
Let $T_1,\ldots, T_d$ be the rooted trees obtained by the deletion of $r$ from $T$, where for each $j\in [r]$, the root of $T_j$ is the unique successor of $r$ in $T_j$.
Then for all $j\in [r]$ we have that $i(T_j) < i(T)$, and by the induction hypothesis every $T_j$ satisfies that $\ell(T_j)\ge i(T_j)+1$.
Observe that every internal node of $T_j$, $j\in [d]$, is also internal in $T$, and node $r$ is internal in $T$ as well, so we have
$i(T)=1+\sum_{j=1}^d i(T_j)$.
Since $\ell(T)=\sum_{j=1}^d\ell(T_j)$, we conclude that
$\ell(T)\ge \sum_{j=1}^d \left(i(T_j)+1\right)= \sum_{j=1}^d i(T_j) + d \ge
(i(T)-1) + 2=i(T)+1$,
which completes the proof.
\end{proof}

\begin{corollary}\label{edges of the modular tree}
Let $T$ be a tree in which each internal node has at least two successors.
Then $|E(T)|\le 2\ell(T)-2$.
\end{corollary}

\begin{proof}
By \cref{leaves in trees}, we have $\ell(T)\ge i(T)+1$.
Therefore, $|V(T)| = i(T) +\ell(T)\le 2\ell(T)-1$ and consequently
$|E(T)| = |V(T)|-1 \le 2\ell(T)-2$.
\end{proof}

\subsection{Modular decomposition}\label{MD}

Given a graph $G$ and a nonempty set $M\subseteq V(G)$, we say that $M$ is a \emph{module} in $G$ if every vertex not in $M$ is either adjacent to all vertices in $M$ or to none of them.
If $M_1$ and $M_2$ are two disjoint modules in a graph $G$, then either $G$ contains all possible edges between $M_1$ and $M_2$ in $G$, or none of them.
A module $M$ is \emph{maximal} if $M\subset V(G)$ and there is no module $M'$ in $G$ with $M\subset M'\subset V(G)$.
If $G$ and its complement are both connected, then any two maximal modules in $G$ are disjoint; in particular, the set of maximal modules of $G$ forms a partition of $V(G)$.
A module $M$ of a graph $G$ is said to be \emph{strong} if for every other module $M'$ in $G$ it holds that either $M\cap M'= \emptyset$, $M\subseteq M'$, or $M'\subseteq M$.
A graph $G$ is \textit{prime} if each of its maximal strong modules is a singleton.

Every graph with at least two vertices has a unique partition of its vertex set into maximal strong modules (see, e.g.,~\cite{habib2010survey}).
If $G$ is disconnected, then the partition is given by the vertex sets of its components; if the complement of $G$ is disconnected, then the partition is given by the vertex sets of its co-components.
The \emph{representative graph} $R(G)$ of $G$ is any induced subgraph of $G$ obtained by taking an arbitrary but fixed vertex from each maximal strong module of $G$.
Note that the representative graph of $G$ depends on how the vertices from the maximal strong modules are chosen; however, any two such graphs are isomorphic to each other, which explains the notation $R(G)$.
The representative graph of $G$ is a special case of the following more general construction.
Given a graph $G$ and an arbitrary partition $\mathcal{P} = \{M_1,\dots, M_k\}$ of $V(G)$ into modules of $G$, we denote by $G/{\mathcal{P}}$ the corresponding \emph{quotient graph}, which is the induced subgraph of $G$ obtained by taking one vertex from each module $M_j\in \mathcal{P}$.

Partitioning the vertex set of a graph $G$ recursively into maximal strong modules leads to the so-called \emph{modular decomposition} of $G$, represented with the so-called \emph{modular decomposition tree}.
This is a rooted tree $T_G$ such that every node of $T_G$ is labeled with an induced subgraph $H_t$ of $G$, and every internal node of $T_G$ is of one of the types \textit{parallel}, \textit{series}, or \textit{prime}.
The tree $T_G$ is defined recursively as follows.
\begin{itemize}
    \item If $G$ is the one-vertex graph, then $T_G$ has one node $t$, labeled with $H_t = G$, and $t$ is the root of $T_G$.
\item Otherwise, $T_G$ is the rooted tree obtained by creating a root node $r$, labeling the root by the representative graph of $G$ (that is, setting $H_r = R(G)$), and joining the root $r$ with edges to the roots of the modular decomposition trees $T_1,\ldots, T_k$ of the subgraphs of $G$ induced by the maximal strong modules $M_1,\ldots, M_k$ of $G$.
The root node of $G$ is of type \emph{parallel} if $G$ is disconnected, \emph{series} if the complement of $G$ is disconnected, and \textit{prime} if both $G$ and its complement are connected.
Each internal node $t$ of $T_G$ with $t\neq r$ belongs to a unique tree $T_i$ and its type in $T_G$ is the same as in~$T_i$.
\end{itemize}
Given a graph $G$, the modular decomposition tree $T_G$ of $G$ can be computed in linear time (see~\cite{MR1687819,MR2500307}).
By construction, for every node $t\in V(T_G)$, the subtree of $T_G$ rooted at $t$ is the modular decomposition tree of the subgraph $G_t$ of $G$ induced by the vertices appearing in the one-vertex subgraphs labeling the leaves of this subtree.
Furthermore, if the node $t$ is of type prime, then the graph $H_t$ labeling the node is a prime graph.

\subsection{Well-covering systems}

A \textit{weighted graph} is a pair $(G,w)$ where $G$ is a graph and $w\in \mathbb{R}^{V(G)}$, that is, $w:V(G)\to \mathbb{R}$ is a real-valued vertex weight function.
Given a weighted graph $(G,w)$ and a set $S\subseteq V(G)$, the weight of $S$ (with respect to $w$) is defined as $w(S)=\sum_{v\in S} w(v)$.
Given a set $S\subseteq V(G)$, we denote by $w_S$ the \emph{restriction of $w$ to $S$}, that is, the function $w_{S}: S\to \mathbb{R}$ defined by setting $w_{S}(v)=w(v)$ for all $v\in S$.

Given a weighted graph $(G,w)$, we say that $w$ is a \emph{well-covered weighting} of $G$ and that $G$ is \emph{$w$-well-covered} if all maximal independent sets in $G$ have the same weight with respect to $w$, that is, for every two maximal independent sets $I$ and $I'$ in $G$, we have $w(I) = w(I')$.
Recall that for every graph $G$, the set $\WCW{G}$ of all well-covered weightings of $G$ forms a vector space over the field of real numbers, called the \emph{well-covered vector space} of $G$.
Since we only work with finite graphs, the well-covered vector space $\WCW{G}$ is always finite-dimensional and thus has a finite basis (an inclusion-wise maximal linearly independent set of vectors); furthermore, all bases of $\WCW{G}$ have the same cardinality, which is referred to as the well-covered dimension of $G$.
Clearly, for every graph $G$, its well-covered dimension is an integer between $0$ and $|V(G)|$.

Well-covered vector spaces of graphs can also be represented using systems of linear equations.
Let $G$ be a graph with $n$ vertices.
Fix an arbitrary ordering $v_1,\ldots, v_n$ of the vertices of $G$ and an arbitrary ordering $I_1,\ldots, I_k$ of all maximal independent sets in $G$.
By definition, a weight function $w:V(G)\to \mathbb{R}$ is a well-covered weighting of $G$ if and only if $w$ satisfies the  following system of ${k\choose 2}$ equations:
\begin{equation}\label{eq:WCW1}
w(I_i)-w(I_{j})=0 \quad \textrm{for any two distinct }i,j\in [k]\textrm{ with }i<j\,.
\end{equation}
To distinguish between vectors of abstract variables of a system and vectors of their concrete real values, we use the following convention throughout the paper.
To each vertex \hbox{$v\in V(G)$} we associate a variable $x_v$, and write the systems of equations using such variable names.
For example, following this convention, the system \eqref{eq:WCW1} corresponds to the following homogeneous linear system over the set of  variables $\{x_v : v\in V(G)\}$:
\begin{equation}\label{eq:WCW-x}
\sum_{v\in I_i}x_v-\sum_{v\in I_j}x_v=0\quad \textrm{for any two distinct }i,j\in [k]\textrm{ with }i<j\,.
\end{equation}
This system can be compactly represented with a single matrix equation
\[Ax = 0_r\]
where $r = {k\choose 2}$, $A\in \mathbb{R}^{r\times n}$ is the coefficient matrix, and the right-hand side $0_r$ is the all-zero vector in $\mathbb{R}^{r}$.
Thus, a column vector $w = (w(v_1),\ldots, w(v_n))^\top\in \mathbb{R}^n$ belongs to the well-covered vector space $\WCW{G}$ if and only if $Aw = 0_r$.

There are many ways to represent the well-covered vector space of a given graph $G$ with a linear system.
For example, a system equivalent to~\eqref{eq:WCW-x} with $k-1$ equations can be obtained by requiring that all maximal independents sets have the same weight as an arbitrary but fixed maximal independent set, say $I_k$:
\begin{equation}\label{eq:WCW2}
\sum_{v\in I_i}x_v-\sum_{v\in I_k}x_v=0\quad \textrm{for all }i\in [k-1]\,.
\end{equation}
Another equivalent system, also with $k-1$ equations, is the following:
\begin{equation}\label{eq:WCW3}
\sum_{v\in I_i}x_v-\sum_{v\in I_{i+1}}x_v=0\quad \textrm{for all }i\in [k-1]\,.
\end{equation}
A \emph{well-covering system} of $G$ is any system $\W$ of linear homogeneous equations over a set $\{x_v : v\in V(G)\}$ of variables indexed by the vertices of $G$ such that a column vector $w = (w(v_1),\ldots, w(v_n))^\top\in \mathbb{R}^n$ belongs to the well-covered vector space $\WCW{G}$ if and only if it satisfies all the equations of the system.
Given a well-covering system $\W$ of $G$, we denote by $|\W|$ the \emph{size} of $\W$, that is, the number of equations in $\W$.
As shown by systems~\eqref{eq:WCW-x} and~\eqref{eq:WCW2}, the same graph can admit well-covering systems of different sizes.

We will soon illustrate these concepts with a concrete example, but first let us discuss two important remarks about properties of well-covering systems.

\medskip
\noindent{\bf A remark on the size of well-covering systems.}
The number of maximal independent sets in an $n$-vertex graph can be exponential in $n$.\footnote{For example, the $2n$-vertex graph consisting of $n$ isolated edges
has $2^n$ maximal independent sets.}
However, using Gaussian elimination it can be shown that any well-covering system of an $n$-vertex graph admits a well-covering subsystem of size at most $n$ (see~\Cref{lem:reduced-well-covering}).

Consider an arbitrary well-covering system $\W$ of an $n$-vertex graph $G$ and let $r$ be the size of $\W$.
Fix an arbitrary ordering of the vertices of $G$ and an arbitrary ordering of the equations in $\W$.
Let $A\in \mathbb{R}^{r\times n}$ be the coefficient matrix of $\W$.
We say that a well-covering system $\W$ is \emph{linearly independent} if the rows of the corresponding matrix $A$ are linearly independent over the field of real numbers.
In this case, the $r$ rows of $A$ form a basis of the orthogonal complement of the vector space $\WCW{G}$, and hence by standard linear algebra we have $r+\wcdim{G} = n$.
In particular, in this case we have $r\le n$, and equality holds if and only if $\wcdim{G} = 0$, that is, the all-zero weighting is the only well-covered weighting of~$G$.

\medskip
\noindent{\bf A remark on the coefficients of well-covering systems.}
Since we consider the well-covered vector space $\WCW{G}$ of a graph $G$ as a vector space over the field of real numbers, any well-covering system of $G$ consists of linear equations involving real numbers as coefficients.
However, it often suffices to work with well-covering systems whose coefficients belong to a particular subset of the set of real numbers.
We say that a well-covering system is \emph{unit} if the matrix of the system has all the coefficients in the set $\{-1,0,1\}$, \emph{integer} if the system consists of linear equations involving only integer coefficients, and \emph{rational} if it consists of linear equations involving only rational coefficients.
Note that systems \eqref{eq:WCW-x}, \eqref{eq:WCW2}, and \eqref{eq:WCW3} are all unit.
Furthermore, the well-covering systems of fork-free graphs constructed by the algorithm given by our main result (\cref{thm:main}) are also unit.

\begin{example}
Let $G$ be the \emph{bull graph}, that is, the graph obtained from the $5$-vertex path with vertices $v_1,\ldots, v_5$ in order along the path by adding to it the edge $v_2v_4$.
Then $G$ has exactly three maximal independent sets: $I_1=\{v_1,v_4\}$, $I_2=\{v_2,v_5\}$, and  $I_3=\{v_1,v_3,v_5\}$.
Any well-covered weighting $w$ of $G$ must satisfy that $w(I_1)=w(I_2)=w(I_3)$, or equivalently, $w(I_1)-w(I_3)=0$ and $w(I_2)-w(I_3)=0$.
This yields the following linearly independent unit well-covering system $\W$ of $G$ with size $r = 2$:
\begin{alignat*}{6}
   &  &   & {}-{} & x_{v_3} & {}+{} & x_{v_4} & {}-{} & x_{v_5} & {}={} & 0\\
  -x_{v_1} & {}+{} & x_{v_2}  & {}-{} & x_{v_3} & & & & & {}={} & 0\,.
\end{alignat*}
Using this system of equations we can easily determine for any weighting $w$ of $G$ whether it is well-covered weighting or not.
For example, letting \[\mathcal{B} = \{(1,1,0,0,0)^\top, (0,1,1,1,0)^\top, (0,0,0,1,1)^\top\}\,,\] it can be easily verified that each $w=(w(v_1),\ldots,w(v_5))^\top\in \mathcal{B}$ satisfies both equations in $\W$ and thus belongs to the space $\WCW{G}$.
Furthermore, since the two rows of the coefficient matrix of the system $\W$, that is, $(0,0,-1,1,-1)$ and $(-1,1,-1,0,0)$, form a basis of the orthogonal complement of the well-covered vector space, it follows that the well-covered dimension of the space $\WCW{G}$ equals $|V(G)|-r = 3$.
Thus, since the vectors in the set $\mathcal{B}$ are linearly independent, we infer that $\mathcal{B}$ is a basis of the well-covered vector space $\WCW{G}$.\hfill$\blacktriangle$
\end{example}

In some of our results, including the reduction based on modular decomposition (\Cref{thm:modular+Gauss}), the following lemma based on Gaussian elimination will be useful.
We denote by $\omega < 2.373$ the matrix multiplication exponent~(see, e.g.,~\cite{MR4262465}).

\begin{sloppypar}
\begin{lemma}\label{lem:reduced-well-covering}
Given an $n$-vertex graph $G$ and a rational well-covering system $\widehat{\W}$ of $G$, one can compute in time $\mathcal{O}(n^{\omega-1}|\widehat{\W}|)$ a linearly independent well-covering system $\W\subseteq \widehat{\W}$ of $G$ such that $|\W|\le \min\{n,|\widehat{\W}|\}$.
\end{lemma}
\end{sloppypar}

\begin{proof}
Let $r= |\widehat{\W}|$.
If $r\le n$, we are done, so assume $r> n$.
Fix an arbitrary ordering of the vertices of $G$ and an arbitrary ordering of the equations in $\widehat{\W}$.
Let $A\in \mathbb{Q}^{r\times n}$ be the corresponding matrix and let $A^\top$ be its transpose.
Using Gaussian elimination, we compute a basis $B$ of $A^\top$ that is a maximal linearly independent subset of columns of $A^\top$.
This can be done in time $\mathcal{O}(rn^{\omega-1})$ (see~\cite{MR3124680}).
Note that the vectors in $B$ correspond to certain equations in $\widehat{\W}$.
Let $\W\subseteq \widehat{\W}$ consist of equations corresponding to the vectors in $B$.
Then $\W$ is a linearly independent well-covering system of $G$, and clearly $|\W|\le \min\{n,r\}$.
Since $\omega \ge 2$ and the matrix $A$ and its transpose can be computed in time $\mathcal{O}(rn)$, the algorithm runs in time $\mathcal{O}(rn^{\omega-1})$.
\end{proof}

\section{Reduction to prime induced subgraphs}\label{sec:prime}

In this section we explain how to efficiently compute a well-covering system of a graph from well-covering systems of its maximal strong modules and of the representative graph.
Then we combine this result with modular decomposition and Gaussian elimination to reduce the problem of computing a well-covering system of a graph to the same problem on certain prime induced subgraphs of the graph.

We start with a basic lemma characterizing the family of maximal independent sets in a graph $G$ whose vertex set is equipped with an arbitrary partition into modules.

\begin{lemma}
\label{lem:ind-set-modules}
Let $G$ be a graph, let $\mathcal{P} = \{M_1,\dots, M_k\}$ be an arbitrary partition of $V(G)$ into modules, and let $G' = G/{\mathcal{P}}$ be the corresponding quotient graph, with $V(G') = \{v_1,\ldots, v_k\}$ where $v_j\in M_j$ for all $j\in [k]$.
Then, a set $X\subseteq V(G)$ is a maximal independent set in $G$ if and only if the following conditions hold.
\begin{itemize}
\item[i)] For all $j\in [k]$, the set $X\cap M_j$ is either empty or a maximal independent set in $G[M_j]$.
\item[ii)] The set $X'=\{v_j\in V(G'):X\cap M_j\neq\emptyset\}$ is a maximal independent set in $G'$.
\end{itemize}
\end{lemma}

\begin{proof}
First we show that the stated conditions are necessary.
Let $X\subseteq V(G)$ be a maximal independent set in $G$.
Consider an arbitrary $j\in [k]$ such that $X\cap M_j\neq \emptyset$.
We want to prove that $X\cap M_j$ is a maximal independent set in $G[M_j]$.
Since this set is a subset of $X$, it is an independent set in $G$ and hence also in $G[M_j]$.
We have to prove that it is a maximal one.
Suppose for a contradiction that this is not the case, and let $x\in M_j\setminus X$ satisfy that $(X\cap M_j)\cup \{x\}$ is an independent set in $G[M_j]$.
Then $x$ has no neighbors in $X\cap M_j$.
Since $X\cap M_j\neq \emptyset$, there exists a vertex $y\in S\cap M_j$.
Note that $x$ and $y$ are in the same module $M_j$, so they have the same neighborhood outside $M_j$ in $G$.
In particular, this implies that $N_G(x)\cap (X\setminus M_j) = N_G(y)\cap (X\setminus M_j) \subseteq N_G(y)\cap X = \emptyset$, where the second equality follows from the fact that $y\in X$ and $X$ is independent in $G$.
We already know that $x$ has no neighbors in $X\cap M_j$, so it follows that $x$ has no neighbors in the set $X$ at all.
This implies that $X\cup \{x\}$ is the independent set in $G$, a contradiction with the maximality of $X$ in $G$.
Hence, condition $i)$ holds.

Next we show condition $ii)$, that is, that the set $X'=\{v_j\in V(G'): X\cap M_j\neq \emptyset\}$ is a maximal independent set in $G'$.
Let $J=\{j\in [k] : X\cap M_j\neq \emptyset\}$.
Vertices in $X$ are pairwise nonadjacent, so the modules $M_j$, $j\in J$, that contain vertices from $X$ are anticomplete to each other in $G$.
By construction of the graph $G'$ it follows that the corresponding vertices $v_j, j\in J$, are pairwise nonadjacent in $G'$, hence $X'$ is an independent set in $G'$.
It remains to prove maximality.
Suppose for a contradiction that there is a vertex $v_\ell\in V(G')\setminus X'$ such that $X'\cup \{v_\ell\}$ is an independent set in $G'$.
Since $v_\ell\notin S'$, it follows that $\ell\notin J$ and thus $X\cap M_\ell=\emptyset$.
However, since $X'\cup \{v_\ell\}$ is an independent set in $G'$, for all $j\in J$ we have that $v_\ell\notin N_{G'}(v_j)$, and it follows that modules $M_\ell$ and $M_j$ are anticomplete to each other in $G$.
Thus we can enlarge the independent set $X$ in $G$ by adding to it any vertex from $M_j$.
This contradicts the fact that $X$ is a maximal independent set in $G$.

The two conditions are also sufficient.
Let $X\subseteq V(G)$ and assume that conditions $i)$ and $ii)$ from the lemma hold.
We will prove that $X$ is a maximal independent set in $G$.
Let $J=\{j\in [k]: X\cap M_j\neq \emptyset\}$.
Note that $X=\bigcup_{j\in J} (X\cap M_j)$.
By condition $i)$ we have that for all $j\in J$ the set $X\cap M_j\subseteq M_j$ is independent in $G_j$ and hence in $G$.
By condition $ii)$ we have that the set $X'=\{v_j\in V(G'):j\in J\}$ is an independent set in $G'$.
It follows that all the modules $M_j$, $j\in J$, are pairwise anticomplete.
Hence the set $X = \bigcup_{j\in J} (X\cap M_j)$ is independent in $G$.
It remains to show maximality.
Suppose for a contradiction that there exists a vertex $v\in V(G)\setminus X$ such that the set $X\cup \{v\}$ is independent in $G$.
Let $\ell\in [k]$ such that $v\in M_\ell$.
Then $(X\cup \{v\})\cap M_\ell$ is an independent set in $G_\ell$, which implies that the set $X\cap M_\ell$ is not a maximal independent set in $G_\ell$.
By condition $i)$ it follows that $X\cap M_\ell= \emptyset$ and thus $\ell\notin J$.
Since the set $X\cup \{v\}$ is independent in $G$, the vertex $v$ has no neighbors in the set $X = \bigcup_{j\in J} (X\cap M_j)$.
As all the sets $M_j$ are modules in $G$, this implies that $v$ has no neighbors in the set $\bigcup_{j\in J}M_j$.
Consequently, the vertex $v_\ell$ corresponding to the module $M_\ell$ in $G'$ has no neighbors in the set $X' = \{v_j: j\in J\}$, in $G'$.
Hence the set $X'\cup \{v_\ell\}$ is independent in $G'$.
Since $v_\ell\not\in X'$, this is a contradiction with the maximality of $X'$, which is given by condition $ii)$.
This shows that the set $X$ is a maximal independent set in $G$.
\end{proof}

We now use \Cref{lem:ind-set-modules} to show how to efficiently compute a well-covering system of a graph from well-covering systems of its maximal strong modules and of the representative graph.
We state the result more generally, for any graph equipped with a partition of the vertex set into modules, since we will later apply this result to various scenarios depending on whether the graph is disconnected (in which case the modules are the vertex sets of its connected components), the complement of the graph is disconnected  (in which case the modules are the vertex sets of its co-components), or the graph and its complement are both connected.

\begin{lemma}
\label{lem:partition-of-g}
Let $G$ be a graph, let $\mathcal{P} = \{M_1,\dots, M_k\}$ be an arbitrary partition of $V(G)$ into modules, and let $G' = G/{\mathcal{P}}$ be the corresponding quotient graph, with $V(G') = \{v_1,\ldots, v_k\}$ where $v_j\in M_j$ for all $j\in [k]$.
Let $\W_j$ be a well-covering system for $G[M_j]$ for all $j\in[k]$ and let $\W'$ be a well-covering system of $G'$.
Let $\mathcal{I} = \{I_j:j\in [k]\}$ be an arbitrary but fixed collection of maximal independent sets $I_j$ in $G[M_j]$ for all $j\in[k]$.
For each equation $s\in \W'$, let us denote by $\rho_\mathcal{I}(s)$ the equation indexed by the vertices of $G$ obtained from $s$ by iterating over all vertices $v_j$ of $G'$ and substituting the variable $x_{v_j}$ corresponding to the vertex $v_j$ with the sum $\sum_{v\in I_j}x_v$ (in particular, the variables corresponding to vertices $v$ of $G$ that do not belong to the union $\bigcup_{j\in [k]}I_j$ appear with zero coefficient).
Then
\begin{equation}\label{well-covering system}
\W = \left(\bigcup_{j=1}^k \W_j \right) \cup \Big\{\rho_{\mathcal{I}}(s): s\in \W'\Big\}
\end{equation}
is a well-covering system of $G$.
Furthermore, if the systems $\W_1,\ldots, \W_k$ and $\W'$ are all rational (resp.~integer or unit), then so is $\W$.
\end{lemma}

\begin{proof}
Let $G_j$ denote the graph $G[M_j]$ for all $j\in[k]$.
The proof of the lemma will be based on the following observation.

\medskip
\noindent \textbf{Claim.}
\textit{Let $w$ be a vertex weight function on $G$, and let $w':V(G')\to \mathbb{R}$ be defined as $w'(v_j)=w(I_j)$ for all $j\in[k]$.
Let also $w_j$ denote the restriction of $w$ to $V(G_j)$ for all $j\in[k]$.
Then $G$ is $w$-well-covered if and only if $G'$ is $w'$-well-covered and for all $j\in [k]$, the graph $G_j$ is $w_j$-well-covered.}

\medskip
Let us first show that the claim implies the lemma.
We show that the proposed system of equations $\W$ given by~\eqref{well-covering system} is a well-covering system of $G$ by showing that, for any vertex weight function $w$ on $G$, it holds that $w$ is a well-covered weighting of $G$ if and only if $w$ satisfies all the equations of the system.
Assume first that $w$ is a well-covered weighting of $G$.
Then, by the claim $G'$ is $w'$-well-covered and for all $j\in [k]$, the graph $G_j$ is $w_j$-well-covered.
Since $G'$ is $w'$-well-covered, $w'$ is a solution of the system of equations $\W'$.
Consider an arbitrary equation $s\in \W'$.
Then there exist real numbers $a_{v_j}$, $j\in [k]$, such that $s$ equals the equation $\sum_{j = 1}^k a_{v_j}x_{v_j}=0$.
Hence, the equation $\rho_{\mathcal{I}}(s)$ is equivalent to the equation $\sum_{j=1}^k a_{v_j}\sum_{v\in I_j} x_{v}=0$.
Since setting $x_{v_j} = \sum_{v\in I_j}w(v)$ for all $v_j\in V(G')$ results in a solution of the equation $s$, we infer that setting $x_{v} = w(v)$ for all $v\in V(G)$ results in a solution of the equation $\rho_{\mathcal{I}}(s)$.
Similarly, for each $j\in [k]$, setting $x_{v} = w_j(v) = w(v)$ for all $v\in V(G_j)$ yields a solution of the system of equations $\W_j$.
It follows that setting $x_{v} = w(v)$ for all $v\in V(G)$ results in a solution of the system of equations $\bigcup_{j = 1}^k\W_j$ and thus of the entire system of equations~\eqref{well-covering system}.
Similar arguments show that if $w$ is a solution of the system of equations~\eqref{well-covering system}, then $w$ is a well-covered weighting of $G$.
The last statement of the lemma, that the system $\W$ is rational (resp.~integer or unit) whenever this is the case for the systems $\W_1,\ldots, \W_k$ and $\W'$, is straightforward.

Now we show the claim.
Assume that $G$ is $w$-well-covered.
First we show that $G'$ is $w'$-well-covered.
Let $I$ and $I'$ be two maximal independent sets in $G'$.
Let $J = \{j\in [k]: v_j\in I\}$ and
$J' = \{j\in [k]: v_j\in I'\}$ be the corresponding index sets.
By Lemma~\ref{lem:ind-set-modules}, the sets $\bigcup_{j\in J} I_j$ and $\bigcup_{i\in J'}I_j$ are maximal independent sets in $G$.
Since $G$ is $w$-well-covered, it follows that $w\big(\bigcup_{j\in J} I_j\big)=w\big(\bigcup_{i\in J'}I_j\big)$.
Furthermore, we have
\[w'(I)=\sum_{j\in J}w'(v_j)=\sum_{j\in J} w(I_j)=w\Bigg(\bigcup_{j\in J} I_j\Bigg)\]
and
\[w'(I')=\sum_{j\in J'}w'(v_j)=\sum_{j\in J'} w(I_j)=w\Bigg(\bigcup_{j\in J'} I_j\Bigg)\,.\]
Altogether, the above equations imply that $w'(I)=w'(I')$ and since $I$ and $I'$ were arbitrary maximal independent sets in $G'$, it follows that $G'$ is $w'$-well-covered.

Next, we show that for all $j\in [k]$, the graph $G_j$ is $w_j$-well-covered.
Let $I$ and $I'$ be arbitrary maximal independent sets in $G_j$, and let $S$ be a maximal independent set in $G'$ such that $v_j\in S$.
Let also $X = \bigcup\{I_\ell: v_\ell\in S\setminus \{v_j\}\}$.
By Lemma~\ref{lem:ind-set-modules}, the sets $I\cup X$ and $I'\cup X$ are maximal independent sets in $G$.
Since $G$ is $w$-well-covered, we have that $w(I\cup X)=w(I'\cup X)$, and consequently
\[w_j(I)=w(I)=w(I\cup X)-w(X)=w(I'\cup X)-w(X)=w(I')=w_j(I')\,.\]
Since $I$ and $I'$ were arbitrary maximal independent sets in $G_j$, we infer that $G_j$ is $w_j$-well-covered.

For the proof of the other direction, assume that $G'$ is $w'$-well-covered and that $G_j$ is $w_j$-well-covered for all $j\in [k]$.
We want to show that $G$ is $w$-well-covered.
Let $I$ and $I'$ be maximal independent sets in $G$, and let $J,J'\subseteq [k]$ be defined as $J=\{j\in [k]: I\cap M_j\neq \emptyset\}$ and $J'=\{j\in [k]: I'\cap M_j\neq \emptyset\}$.
By Lemma~\ref{lem:ind-set-modules}, the sets $S=\{v_j\in V(G') : j\in J\}$ and $S'=\{v_j\in V(G') : j\in J'\}$ are maximal independent sets in $G'$, and for all $j\in J$ (resp.~$j\in J'$), the set $I\cap M_j$ (resp.~$I'\cap M_j$) is a maximal independent set in $G_j$.
Since for all $j\in [k]$ we have that $G_j$ is $w_j$-well-covered, it follows that
\[w(I\cap M_j)=w_j(I\cap M_j)=w_j(I_j)=w(I_j)=w'(v_j) {\rm  \,\,\, for \, all \,}j\in J\] and similarly \[w(I'\cap M_j)=w_j(I'\cap M_j)=w_j(I_j)=w(I_j)=w'(v_j) {\rm \,\,\, for \, all \,}j\in J'.\]
Thus, we have that
$w(I)=\sum_{j\in J} w(I\cap M_j)=\sum_{j\in J} w'(v_j)=w'(S)$
and
$w(I')=\sum_{j\in J'} w(I'\cap M_j)=\sum_{j\in J'} w'(v_j)=w'(S')$.
Since $G'$ is $w'$-well-covered, it follows that $w'(S)=w'(S')$ and consequently $w(I)=w(I')$, as we wanted to show.
The sets $I$ and $I'$ were arbitrary maximal independent sets in $G$, hence it follows that $G$ is $w$-well-covered.
\end{proof}

We now apply \Cref{lem:partition-of-g} to three different cases: when $G$ is disconnected, when the complement of $G$ is disconnected, and when both $G$ and its complement are connected.

\begin{corollary}
\label{cor:disjoint-union}
Let $G$ be a disconnected graph, with connected components $G_1,\dots, G_k$ for some $k\ge 2$, and let $\W_j$ be a well-covering system of $G_j$ for all $j\in [k]$.
Then $\W = \bigcup_{j = 1}^k\W_j$ is a well-covering system of $G$ that can be computed in time $\mathcal{O}\big(\sum_{j = 1}^k|\W_j|\big)$.
Furthermore, if the systems $\W_1,\ldots, \W_k$ are all rational (resp.~integer or unit), then so is $\W$.
\end{corollary}

\begin{sloppypar}
\begin{proof}
Let $G$ be a graph with connected components $G_1,\ldots, G_k$.
Then $\mathcal{P}=\{V(G_1), \ldots, V(G_k)\}$ is a partition of $V(G)$ into modules, and the corresponding quotient graph $G'=G/\mathcal{P}$ is the edgeless graph with $k$ vertices.
This implies that $V(G')$ is the only maximal independent set in $G'$ and hence $\mathcal{S}'=\emptyset$ is a well-covering system of $G'$.
By \cref{lem:partition-of-g}, it follows that the set $\bigcup_{j=1}^k \W_j $ is a well-covering system of~$G$.
This system can be computed in time $\mathcal{O}\big(\sum_{j = 1}^k|\W_j|\big)$.
\end{proof}
\end{sloppypar}

\Cref{cor:disjoint-union} implies the fact that the well-covered dimension of a graph is the sum of the well-covered dimensions of its connected components (see~\cite{zbMATH05029664}).

\begin{sloppypar}
\begin{corollary}
\label{cor:join}
Let $G$ be a graph with disconnected complement, with co-components $G_1,\dots, G_k$, for some $k\ge 2$, and let $\W_j$ be a well-covering system of $G_j$ for all $j\in[k]$.
Let $I_j$ be a maximal independent set in $G_j$ for $j\in[k]$.
Then
 \[\W = \left(\bigcup_{j=1}^k \W_j\right) \cup \left\{\sum_{v\in I_j}x_v - \sum_{v\in I_{j+1}}x_v=0 : j\in [k-1]\right\}\]
 is a well-covering system of $G$.
In particular, given $G$, $G_1,\ldots, G_k$, and $\W_1,\ldots, \W_k$ as above, a well-covering system of $G$ with size $\sum_{j = 1}^k|\W_j|+k-1$ can be computed in time $\mathcal{O}(|V(G)|+|E(G)|+\sum_{j = 1}^k|\W_j|)$.
Furthermore, if the systems $\W_1,\ldots, \W_k$ are all rational (resp.~integer or unit), then so is $\W$.
\end{corollary}
\end{sloppypar}

\begin{sloppypar}
\begin{proof}
Let $G$ be a graph with co-components $G_1,\ldots, G_k$.
Then $\mathcal{P}=\{V(G_1), \ldots, V(G_k)\}$ is a partition of $V(G)$ into modules, and the corresponding quotient graph $G'=G/\mathcal{P}$ is the complete graph on $k$ vertices.
Let $V(G')=\{v_1,\ldots, v_k\}$.
Since $G'$ is complete graph, the maximal independent sets in $G'$ are exactly the singletons $\{v_j\}$ for $j\in [k]$.
Consequently, $w'$ is a well-covered weighting of $G'$ if and only if $w'(v_1)=\ldots=w'(v_k)$, or equivalently, if for all $j\in [k-1]$ we have that $w'(v_j)=w'(v_{j+1})$.
It follows that the set $\mathcal{S'}=\{x_{v_j}-x_{v_{j+1}}=0: j\in [k-1]\}$ is a well-covering system of $G'$.
Let $\mathcal{I} = \{I_j: j\in [k]\}$.
We follow the notation from \Cref{lem:partition-of-g} and for each $s\in \mathcal{S'}$ denote by $\rho_\mathcal{I}(s)$ the equation indexed by the vertices of $G$ obtained from $s$ by replacing each variable $x_{v_j}$ corresponding to a vertex $v_j$ of $G'$ with the sum $\sum_{v\in I_j}x_v$.
By Lemma~\ref{lem:partition-of-g} it follows that $\left(\bigcup_{j=1}^k \W_j \right) \cup \Big\{\rho_{\mathcal{I}}(s): s\in \W'\Big\}$ is a well-covering system of $G$.
Thus, the set $\Big\{\rho_{\mathcal{I}}(s): s\in \W'\Big\}$ is equivalent to the set $\left\{\sum_{v\in I_j}x_v - \sum_{v\in I_{j+1}}x_v=0 :j\in [k-1]\right\} $.
It follows that $\W = \left(\bigcup_{j=1}^k \W_j\right) \cup \left\{\sum_{v\in I_j}x_v - \sum_{v\in I_{j+1}}x_v=0 : j\in [k-1]\right\}$ is a well-covering system of $G$, as claimed.
Furthermore, this system is integer, resp.~unit, if the systems $\W_1,\ldots, \W_k$ are integer, resp.~unit.

It remains to justify the time complexity.
First, we compute for all $j\in [k]$ a maximal independent set $I_j$ in the graph $G_j$.
This can be done using a straightforward greedy algorithm in time $\mathcal{O}(\sum_{j = 1}^k (|V(G_j)|+|E(G_j)|)) = \mathcal{O}(|V(G)|+|E(G)|)$.
We compute the system of equations $\bigcup_{j=1}^k \W_j$ in time  $\mathcal{O}(\sum_{j=1}^k |\W_j|)$ and the system of equations
$\left\{\sum_{v\in I_j}x_v - \sum_{v\in I_{j+1}}x_v=0 : j\in [k-1]\right\}$
in time $\mathcal{O}\Big(\sum_{j = 1}^{k-1}(|I_{j+1}|+|I_j|)\Big)
= \mathcal{O}\Big(\sum_{j = 1}^{k}|I_j|\Big)
= \mathcal{O}(|V(G)|)$.
The total time complexity is $\mathcal{O}(|V(G)|+|E(G)|+\sum_{j = 1}^k|\W_j|)$, as claimed.
\end{proof}
\end{sloppypar}

In the case when the graph and its complement are both connected, the corresponding algorithmic consequence of \Cref{lem:partition-of-g} is as follows.

\begin{corollary}
\label{cor:modular-decomposition}
Let $G = (V,E)$ be a connected and co-connected graph, let $\{M_1,\dots, M_k\}$ be the partition of $V(G)$ into maximal strong modules, and let $G'$ be the representative graph of $G$.
Let $I_j$ be a maximal independent set in the graph $G[M_j]$, let $\W_j$ be a well-covering system for $G[M_j]$ for all $j\in[k]$, and let $\W'$ be a well-covering system of $G'$.
Then a well-covering system $\W$ of $G$ with size $\sum_{j = 1}^k|\W_j|+|\W'|$ can be computed in time $\mathcal{O}\big(|V|\cdot |\W'|+\sum_{j=1}^k |\W_j|\big)$.
Furthermore, if the systems $\W_1,\ldots, \W_k$ and $\W'$ are all rational (resp.~integer or unit), then so is $\W$.
\end{corollary}

\begin{proof}
Let $\mathcal{I} = \{I_j:j\in [k]\}$.
Using the notation of \Cref{lem:partition-of-g}, the lemma implies that it suffices to compute the system of equations $\W = \big(\bigcup_{j=1}^k \W_j \big) \cup \big\{\rho_{\mathcal{I}}(s): s\in \W'\big\}\,.$
This can be done in time
\[\mathcal{O}\left(\sum_{j=1}^k |\W_j|+|\W'|\left(\sum_{j = 1}^k|I_j|\right)\right) =
\mathcal{O}\left(\sum_{j=1}^k |\W_j|+|V|\cdot|\W'|\right)\,,\]
as claimed.
\end{proof}

We now prove the main result of this section, a reduction of the problem of computing a well-covering system of a graph to the same problem on certain prime induced subgraphs of the graph.
We say that a function $f:\mathbb{R}^+\times \mathbb{R}^+ \to \mathbb{R}^+$ is \emph{nondecreasing}
if $0\le x_1\le x_2$ and $0\le y_1\le y_2$ implies $f(x_1,y_1)\le f(x_2,y_2)$, and \emph{superadditive} if the inequality
 \[ f(x_1,y_1)+f(x_2,y_2) \le f(x_1+x_2,y_1+y_2)\]
holds for all $x_1,y_1,x_2,y_2\in \mathbb{R}^+$.
Note that every superadditive function is nondecreasing.

\begin{theorem}\label{thm:modular+Gauss}
Let $\G$ be a class of graphs and $\G^*$ the class of all prime induced subgraphs of graphs in $\G$.
Assume that for each graph $G$ in $\G^*$ with $n$ vertices and $m\ge 1$ edges one can compute in time $f(n,m)$ a rational (resp.~integer or unit) well-covering system of $G$ with size at most $n$, where $f$ is a superadditive function.
Then for any graph $G$ in $\mathcal{G}$ with $n$ vertices and $m$ edges, one can compute in time
$\mathcal{O}\big(f(2n, m) +n^{\omega+1}\big)$ a rational (resp.~integer or unit) well-covering system of $G$ with size at most~$n$.
\end{theorem}

\begin{proof}
Let $G$ be a graph in $\G$ with $n$ vertices and $m$ edges.
Let $T_G$ be the modular decomposition tree of $G$.
This tree can be computed in time $\mathcal{O}(n+m)$~\cite{MR1687819,MR2500307}.
Recall that for a node $t$ of $T_G$, we denote by $G_t$ the subgraph of $G$ induced by the vertices appearing in the one-vertex subgraphs labeling the leaves of the subtree of $T_G$ rooted at~$t$. Let $n_t = |V(G_t)|$ and $m_t = |E(G_t)|$.

We traverse the tree $T_G$ bottom-up and for each node $t\in V(T_G)$ we recursively compute a maximal independent set $I_t$ in $G_t$ and a well-covering system $\W_t$ of $G_t$ with size at most $n_t$.
It is important to note that we do not store a complete representation of the graph $G_t$ via adjacency lists, as that would additionally increase the time and space complexity of the procedure.
The ordering in which the nodes of tree $T_G$ are traversed can be computed in time $\mathcal{O}(|V(T_G)|) = \mathcal{O}(n+m)$, for example, by reversing the ordering in which the nodes of $T_G$ are visited by a breadth-first search from the root node.
For each node $t$ of $T_G$, we denote by $C_t$ the set of all children of $t$ in $T_G$.

Assume first that $t$ is a leaf node (that is, $C_t = \emptyset$).
Then $V(G_t) = \{v_t\}$ where $v_t$ is the vertex of $G$ labeling $t$; in particular, $n_t = 1$.
Hence, $I_t = V(G_t)$ is the only maximal independent set in $G_t$ and $\W_t=\emptyset$ is a well-covering system of $G_t$ that trivially satisfies the inequality $|\W_t|\le n_t$.
Both $I_t$ and $\W_t$ can be computed in constant time.

Assume now that $t$ is an internal node in $T_G$.
Then $t$ is one of the types parallel, series, or prime.
Since the subtrees of $T_G$ rooted at the children of $t$ are the modular decomposition trees of the subgraphs of $G_t$ induced by its maximal strong modules, which form a partition of $V(G_t)$, it follows that $n_t = \sum_{u\in C_t}n_u$.
For each child $u$ of $t$ we have already computed a maximal independent set $I_u$ in $G_u$ and a well-covering system $\W_u$ of $G_u$ with size at most $n_u$.
We explain how to efficiently combine these into a maximal independent set $I_t$ in $G_t$ and a well-covering system $\W_t$ of $G_t$ with size at most $ n_t$ for each of the three cases separately.
\begin{itemize}
\item If $t$ is of type parallel, then $G_t$ is a disconnected graph, with connected components $G_{u}$, $u\in C_t$.
We can thus take $I_t = \bigcup_{u\in C_t}I_u$ and by  \cref{cor:disjoint-union},
$\W_t = \bigcup_{u\in C_t}\W_u$.
We have
\[|\W_t| = \sum_{u \in C_t} |\W_{u}|\le \sum_{u\in C_t} n_u=n_t.\]
Furthermore, by \cref{cor:disjoint-union} the well-covering system $\W_t$ of $G_t$ can be computed in time $\mathcal{O}\big(\sum_{u \in C_t} |\W_{u}|\big) = \mathcal{O}\big(|\W_t|\big) = \mathcal{O}(n_t)$.
Since $I_t$ can be computed in time $\mathcal{O}(|V(G_t)|+|E(G_t)|)=\mathcal{O}(n_t+m_t)$, the total time complexity at the parallel node $t$ is $\mathcal{O}(n_t+m_t)$.

\item If $t$ is of type series, then the complement of $G_t$ is disconnected, with co-components $G_{u}$, $u\in C_t$.
We select an arbitrary $u\in C_t$ and set $I_t=I_u$.
Furthermore, we fix an arbitrary ordering $u_1,\ldots, u_p$ of the set $C_t$ and set
\[\widehat{\W_t} = \left(\bigcup_{u\in C_t}\W_u\right) \cup \left\{\sum_{v\in I_{u_j}}x_v - \sum_{v\in I_{u_{j+1}}}x_v=0 : j\in [p-1]\right\}\,.\]
By \cref{cor:join}, $\widehat{\W_t}$ is a well-covering system of $G_t$ that can be computed in time $\mathcal{O}(|V(G_t)|+|E(G_t)| +\sum_{u\in C_t} |\W_u|) = \mathcal{O}(n_t+m_t)$.
The size of $\widehat{\W_t}$ is bounded as follows:
\[
|\widehat{\W_t}| = \sum_{u\in C_t}|\W_u|+|C_t|-1
\le \sum_{u\in C_t}n_u + n_t-1=n_t+n_t-1=2n_t-1.
\]
Furthermore, \cref{lem:reduced-well-covering} implies that a well-covering system $\W_t\subseteq \widehat{\W_t}$ of $G_t$ such that $|\W_t|\le n_t$ can be computed in time $\mathcal{O}(n_t^{\omega-1}\cdot|\widehat{\W_t}|) = \mathcal{O}(n_t^{\omega})$.
Altogether, this implies that the independent set $I_t$ and a well-covering system $\W_t$ of $G_t$ with size at most $n_t$ at the series node $t$ can be computed in time $\mathcal{O}(n_t+m_t+ n_t^{\omega}) = \mathcal{O}(n_t^\omega)$ (since $\omega\ge 2$).

\item Consider now the case when the node $t$ is of type prime.
In this case, the graph $H_t$ labeling the node $t$ is a prime induced subgraph of $G_t$ and hence of $G$.
Each child $u$ of $t$ in $T_G$ corresponds to a unique maximal strong module $M_u$ of $G$.
The graph $H_t$ is the representative graph of $G_t$, hence it contains a unique vertex $v_u$ from each maximal strong module $M_u$ of $G_t$.

\begin{sloppypar}
Since $H_t$ is a prime induced subgraph of $G$, it belongs to $\G^*$ and hence, a well-covering system $\W'$ of $H_t$ with size at most $|V(H_t)|$ can be computed in time $f(|V(H_t)|,|E(H_t)|)$.
Next, we compute in time $\mathcal{O}(|V(H_t)|+|E(H_t)|)$ a maximal independent set $I_t'$ in $H_t$.
Let $C_t' = \{u\in C_t: v_u\in I_t'\}$.
By \cref{lem:ind-set-modules}, the set $I_t=\bigcup_{u\in C_t'} I_{u}$ is a maximal independent set in $G_t$.
By \cref{cor:modular-decomposition}, a well-covering system $\widehat{\W_t}$ of $G_t$ with size $\sum_{u\in C_t}|\W_u|+|\W'|$ can be computed in time $\mathcal{O}\big(|V(G_t)|\cdot |\W'|+\sum_{u\in C_t} |\W_u|\big)$.
Since $|\W'|\le |V(H_t)|\le n_t$, it follows that $|\widehat{\W_t}|\le \sum_{u\in C_t} n_u + n_t=n_t+n_t=2n_t.$
\end{sloppypar}

\begin{sloppypar}
Using \cref{lem:reduced-well-covering}, a well-covering system $\W_t\subseteq \widehat{\W_t}$ of $G_t$ such that $|\W_t|\le n_t$ can be computed in time $\mathcal{O}\big(n_t^{\omega-1}|\widehat{\W_t}|\big)=\mathcal{O}(n_t^\omega)$.
The total time complexity of computing $\W_t$ at the node $t$ is
\begin{align*}
&\mathcal{O}\left(f(|V(H_t)|, |E(H_t)|) + |V(G_t)|\cdot |\W'|+\sum_{u\in C_t} |\W_u| + n_t^\omega\right)\\
=\,&\mathcal{O}(f(|V(H_t)|, |E(H_t)|) + n_t^2 + n_t + n_t^\omega)  \\
=\,&\mathcal{O}( f(|V(H_t)|, |E(H_t)|) + n_t^\omega),
\end{align*}
while the independent set $I_t$ can be computed in time
$\mathcal{O}(|V(H_t)|+|E(H_t)|+|V(G_t)|) =\mathcal{O}(n_t+m_t).$
\end{sloppypar}

Thus, the total time complexity at the prime node $t$ is
$\mathcal{O}(f(|V(H_t)|, |E(H_t)|)  +n_t^\omega) $.
\end{itemize}

It remains to sum up the time complexities over all nodes of $T_G$.
We compute separately the sum over all leaves of $T_G$ and over all internal nodes of $T_G$.
Let us denote by $L$ the set of all leaves of $T_G$.
Recall that by the definition of a modular decomposition tree, the leaves of $T_G$ are in a bijective correspondence with the vertices of $G$, and thus $|L|=n$.
By \cref{leaves in trees} it follows that the number of internal nodes of $T_G$ is at most $n-1$.
Note also that for each internal node $t$, the number of vertices of $H_t$ equals the number of children of $t$ in $T_G$, which implies that the total number of vertices of the graphs $H_t$, summed up over all internal nodes $t$, equals the number of edges of $T_G$, which is at most $|L|+|V(T_G)\setminus L|-1 \le n+(n-1)-1 = 2n-2$.
Furthermore, for each internal node $t$, the edges of $H_t$ correspond to distinct edges of $G$ (joining two vertices of $G_t$ from distinct maximal strong modules), and no two edges from representative graphs of two different internal nodes correspond to the same edge of $G$.
This implies that the total number of edges of the graphs $H_t$, summed up over all internal nodes $t$, is at most $m$.

We already saw that in each leaf $t$ of $T_G$ the algorithm computes the independent set $I_t$ and the well-covering system $\W_t$ in constant time.
Hence, summing over all leaves of $T_G$ we obtain the time complexity of $\mathcal{O}(n)$.
If $t$ is an internal node, then the algorithm computes $I_t$ and $\W_t$ in time $\mathcal{O}(n_t+m_t)$ if $t$ is of type parallel, in time $\mathcal{O}(n_t^\omega)$ if $t$ is of type series, and in time
$\mathcal{O}(f(|V(H_t)|, |E(H_t)|)  +n_t^\omega) $
if $t$ is of type prime.
Furthermore, $|E(H_t)|\le m$.

\begin{sloppypar}
The sum of time complexities over all the internal nodes of $T_G$ can thus be bounded as follows.
\begin{align*}
&\mathcal{O}\Bigg(\sum_{t\in V(T_G)\setminus L} \bigg(  f(|V(H_t)|, |E(H_t)|) + n_t^\omega\bigg)\Bigg)\\
=\, &\mathcal{O}\Bigg(   f\Bigg(\sum_{t\in V(T_G)\setminus L}|V(H_t)|, \sum_{t\in V(T_G)\setminus L}|E(H_t)|\Bigg) +\sum_{t\in V(T_G)\setminus L} n_t^\omega\Bigg)\,=\\
=\, & \mathcal{O}\Bigg( f(2n, m) + n^{\omega+1}\Bigg)\,,
\end{align*}
where the first equality holds since $f$ is a superadditive function and the last one since \hbox{$\sum_{t\in V(T_G)\setminus L}|V(H_t)|\le 2n-2$,} $\sum_{t\in V(T_G)\setminus L}|E(H_t)|\le m$, and $f$ is nondecreasing.
Since the time complexity over all leaves of $T_G$ is $\mathcal{O}(n),$ the total time complexity over all nodes in $T_G$ is equal to
$\mathcal{O}\big(f(2n,m) + n^{\omega+1}\big)$.
Finally, recall that the algorithm first needs $\mathcal{O}(n+m)$ time to compute the modular decomposition tree $T_G$ and an ordering in which the nodes of $T_G$ are visited.
Thus, altogether, the algorithm runs in time
$\mathcal{O}\big(n+m +f(2n,m) + n^{\omega+1}\big) = \mathcal{O}\big(f(2n,m) + n^{\omega+1}\big)$.
\end{sloppypar}
\end{proof}

\begin{remark}\label{remark:modular}
One of the assumptions in \Cref{thm:modular+Gauss} is that for each graph $G$ in $\G^*$ with $n$ vertices and $m$ edges one can compute  in time $f(n,m)$ a well-covering system of $G$ with size at most $n$.
If instead, only an algorithm is available for computing an arbitrary rational (resp.~integer or unit) well-covering system of $G\in \G^*$ in time $f(n,m)$ (that is, without a bound of $n$ on the size of the system), then one can first combine such an algorithm with \Cref{lem:reduced-well-covering}.
This would result in an algorithm that, given a graph $G$ from $\mathcal{G}$ with $n$ vertices and $m$ edges, in time $\mathcal{O}\big(f(2n, m)\cdot n^{\omega-1} +n^{\omega+1}\big)$ computes a rational (resp.~integer or unit) well-covering system of $G$ with size at most $n$.
\end{remark}

\section{Cographs}\label{sec:cographs}

The proof of \cref{thm:modular+Gauss} relies on Gaussian elimination.
If the input graph possesses some additional combinatorial structure, the use of Gaussian elimination may be avoided, and this can lead to faster algorithms.
As we show in this section, this is the case for the class of \emph{cographs}.
Cographs are defined as graphs that can be constructed starting from copies of the one-vertex graph using the operations of disjoint union and complementation (see, e.g.,~\cite{zbMATH01289516}).
Thus, the only prime cograph is the one-vertex graph, and the modular decomposition tree of a cograph contains only parallel and series nodes.

Cographs are known to be exactly the $P_4$-free graphs, that is, graphs that contain no $4$-vertex path as an induced subgraph (see, e.g.,~\cite{MR619603}).
In particular, every cograph is fork-free.
Therefore, it follows from \cref{thm:modular+Gauss} that a well-covering system of a given cograph $G$ with $n$ vertices and $m$ edges can be computed in time $\mathcal{O}(n^{\omega+1})$.
We improve this time complexity as follows.

\begin{theorem}\label{thm:cographs}
Given a cograph $G$ with $n$ vertices and $m$ edges, an integer well-covering system of $G$ with size at most $n-1$ can be computed in time $\mathcal{O}(n(n+m))$.
\end{theorem}

\begin{proof}
Let $G$ be a cograph with $n$ vertices and $m$ edges.
Let $T_G$ be the modular decomposition tree of $G$.
As before, given a node $t\in V(T_G)$, we denote by $G_t$ the subgraph of $G$ induced by the vertices appearing in the one-vertex subgraphs labeling the leaves of the subtree of $T_G$ rooted at $t$.
Let $n_t = |V(G_t)|$ and $m_t = |E(G_t)|$.
Since $G$ is a cograph, every internal node of $T_G$ is of type either parallel or series.
We traverse the tree $T_G$ bottom-up and for each node $t\in V(T_G)$ we recursively compute a maximal independent set $I_t$ in $G_t$ and a well-covering system $\W_t$ of $G_t$ with size at most $n_t-1$.
For each node $t$ of $T_G$, we denote by $C_t$ the set of all children of $t$ in $T_G$.

If $t$ is a leaf node (that is, $C_t = \emptyset$), then $I_t = V(G_t)$ is a maximal independent set of $G_t$ and $\W_t=\emptyset$ is a well-covering system of $G_t$, with size $0 = n_t-1$.
Both $I_t$ and $\W_t$ can be computed in constant time.
If $t$ is an internal node in $T_G$, then $t$ is of type either parallel or series.
For each child $u$ of $t$ we have already computed a maximal independent set $I_u$ in $G_u$ and a well-covering system $\W_u$ of $G_u$ with size at most $n_u-1$.
We explain how to efficiently combine these into a maximal independent set $I_t$ in $G_t$ and a well-covering system $\W_t$ of $G_t$ with size at most $ n_t-1$ for both cases.
If $t$ is of type parallel, then $G_t$ is a disconnected graph, with connected components $G_{u}$, $u\in C_t$.
We can thus take $I_t = \bigcup_{u\in C_t}I_u$ and by  \cref{cor:disjoint-union},
$\W_t = \bigcup_{u\in C_t}\W_u$.
We have
\[|\W_t| = \sum_{u \in C_t} |\W_{u}|\le \sum_{u\in C_t} (n_u-1)=n_t-|C_t|\le n_t-1.\]
Furthermore, by \cref{cor:disjoint-union} the well-covering system $\W_t$ of $G_t$ can be computed in time $\mathcal{O}\big(\sum_{u \in C_t} |\W_{u}|\big) = \mathcal{O}\big(|\W_t|\big) = \mathcal{O}(n_t)$.
Since $I_t$ can be computed in time $\mathcal{O}(|V(G_t)|+|E(G_t)|)=\mathcal{O}(n_t+m_t)$, the total time complexity at the parallel node $t$ is $\mathcal{O}(n_t+m_t)$.
If $t$ is of type series, then the complement of $G_t$ is disconnected, with co-components $G_{u}$, $u\in C_t$.
We fix an arbitrary ordering $u_1,\ldots, u_p$ of the set $C_t$ of children of $t$ and obtain the new maximal independent set $I_t$ and a well-covering system $\W_t$ by setting $I_t = I_{u_1}$ and
\[{\W_t} = \left(\bigcup_{u\in C_t}\W_u\right) \cup \left\{\sum_{v\in I_{u_j}}x_v - \sum_{v\in I_{u_{j+1}}}x_v=0 : j\in [p-1]\right\}\,.\]
By \cref{cor:join}, the system ${\W_t}$ is indeed a well-covering system of $G_t$ and can be computed in time $\mathcal{O}(|V(G_t)|+|E(G_t)| +\sum_{u\in C_t} |\W_u|) = \mathcal{O}(n_t+m_t)$.
The size of ${\W_t}$ is bounded as follows:
\[
|\W_t| = \sum_{u\in C_t}|\W_u|+|C_t|-1
\le \sum_{u\in C_t}(n_u-1) + |C_t|-1=n_t-1.
\]
Altogether, this implies that the independent set $I_t$ and a well-covering system $\W_t$ of $G_t$ with size at most $n_t-1$ at the series node $t$ can be computed in time $\mathcal{O}(n_t+m_t)$.

Note that all the well-covering systems computed by the algorithm are integer.
It remains to estimate the time complexity of the algorithm.
The tree $T_G$ can be computed in time $\mathcal{O}(n+m)$~\cite{MR1687819,MR2500307}, and in the same time we can compute an ordering in which the nodes of tree $T_G$ are traversed.
Recall that the number of leaves of $T_G$ is equal to $n$, while from \cref{leaves in trees} it follows that the number of internal nodes of $T_G$ is at most $n-1$.
We already saw that in each leaf $t$ of $T_G$ the algorithm spends only constant time, while in each internal node $t$ of $T_G$ the independent set $I_t$ and a well-covering system $\W_t$ of $G_t$ can be computed in time $\mathcal{O}(n_t+m_t)$.
Summing over all nodes of $T_G$ we get the time complexity
$\mathcal{O} (n+(n-1)\cdot (n+m))=\mathcal{O}(n(n+m)).$
We infer that the total time complexity of the algorithm is $\mathcal{O}(n(n+m))$.
\end{proof}

Let us mention two consequences of \Cref{thm:cographs}.

First, applying the theorem to a given $n$-vertex cograph $G$, we obtain in polynomial time an integer well-covering system $\widehat{\W}$ with size at most $n-1$.
Using Gaussian elimination (cf.~\Cref{lem:reduced-well-covering}), we can then compute in time $\mathcal{O}(n^{\omega})$ a linearly independent well-covering subsystem $\W\subseteq \widehat{\W}$ of $G$.
Consequently, we can compute the well-covered dimension of $G$ as the difference $n-|\W|$.
This implies a result of Brown and Nowakowski~\cite{zbMATH05029664} that the well-covered dimension of cographs can be computed in polynomial time.

Second, a graph $G$ has well-covered dimension equal to zero if and only if the only well-covered weighting of $G$ is the identically zero function, or, equivalently, if $G$ admits no well-covering system with size less than $n$.
Therefore, \Cref{thm:cographs} implies the following.

\begin{corollary}
Every cograph has a strictly positive well-covered dimension.
\end{corollary}

An alternative proof of this result could be obtained by using the fact that every cograph is equistable (see~\cite{MR1268776}).

\section{Reduction to anti-neighborhoods}\label{sec:anti-neighborhoods}

In this section we focus on the subgraphs of a given graph obtained by deletion of the closed neighborhood of some vertex in the graph.
Given a graph $G$ with vertex set $\{v_1,\ldots, v_n\}$, we denote by $G_j$ the graph $G-N[v_j]$, for all $j\in [n]$.
We first show that, given a well-covering system of the graph $G_j$, for all $j\in [n]$, we can efficiently compute a well-covering system of~$G$.

\begin{lemma}
\label{lem:equations-anti-neighborhood}
Let $G$ be a graph with vertex set $\{v_1,\dots, v_n\}$.
For each $j\in [n]$ let $\W_j$ be a rational (resp.~integer or unit) well-covering system of $G-N[v_j]$ and $I_j$ a maximal independent set of $G-N[v_j]$.
Then
\[\left(\bigcup_{j=1}^n \W_j\right) \cup \left\{\sum_{v\in I_j\cup \{v_j\}}x_v - \sum_{v\in I_{j+1}\cup \{v_{j+1}\}}x_v=0 :  j\in [n-1]\right\}\]
is a rational (resp.~integer or unit) well-covering system of $G$.
\end{lemma}

\begin{proof}
Let $G$ be a graph and let $w$ be a vertex weight function on $G$.
For each $j\in [n]$ let $G_j$ denote the graph $G-N[v_j]$ and $w_j$ the restriction of $w$ to $V(G_j)$.
We show the following claim:
$G$ is $w$-well-covered if and only if for all $j\in[n]$ it holds that $G_j$ is $w_j$-well-covered and for all $j\in [n-1]$ it holds that $w(I_j\cup\{v_j\})=w(I_{j+1}\cup\{v_{j+1}\})$.
From the claim we get that the equations from the well-covering systems $\W_j$ of $G_j$, over all $j\in [n]$, along with the equations of the form \[\sum_{v\in I_j\cup \{v_j\}}x_v - \sum_{v\in I_{j+1}\cup \{v_{j+1}\}}x_v=0\] for $j\in [ n-1]$, form a well-covering system of $G$.

Let us prove the claim.
Assume first that $G$ is $w$-well-covered.
Let $j\in[n]$ and let $I$ and $I'$ be maximal independent sets in $G_j$.
Then the sets $I\cup \{v_j\}$ and $I'\cup \{v_j\}$ are maximal independent sets in $G$.
Since $G$ is $w$-well-covered, it holds that $w(I\cup \{v_j\})=w(I'\cup \{v_j\})$.
Consequently, we have that
\[w_j(I)=w(I)=w(I\cup \{v_j\})-w(v_j)=w(I'\cup \{v_j\})-w(v_j)=w(I')=w_j(I')\,,\]
and $G_j$ is $w_j$-well-covered.
Consider now an arbitrary $j\in [n-1]$.
Since $I_j$ and $I_{j+1}$ are maximal independent sets in $G_j$ and $G_{j+1}$, respectively, the sets $I_j\cup \{v_j\}$ and $I_{j+1}\cup \{v_{j+1}\}$ are maximal independent sets in $G$.
Since $G$ is $w$-well-covered, it follows that $w(I_j\cup \{v_j\})=w(I_{j+1}\cup \{v_{j+1}\})$, which is what we wanted to show.

For a proof of the other direction, assume that for all $j\in [n]$ it holds that $G_j$ is \hbox{$w_j$-well-covered} and for all $j\in [ n-1]$ it holds that $w(I_j\cup\{v_j\})=w(I_{j+1}\cup\{v_{j+1}\})$.
In particular, this implies that $w(I_j\cup\{v_j\})=w(I_k\cup\{v_k\})$ for all $j,k \in [n]$.
We want to prove that $G$ is $w$-well-covered.
Let $I$ and $I'$ be maximal independent sets in $G$ and let $v_j\in I$ and $v_k\in I'$.
Note that $I\setminus \{v_j\}$ and $I'\setminus \{v_k\}$ are maximal independent sets in $G_j$ and $G_k$, respectively.
By assumption $G_j$ is \hbox{$w_j$-well-covered} and $G_k$ is \hbox{$w_k$-well-covered}, and thus we have that $w(I\setminus \{v_j\}) = w_j(I\setminus \{v_j\}) = w_j(I_j) = w(I_j)$ and, similarly, $w(I'\setminus \{v_k\})= w(I_k)$.
Consequently,
\[w(I)=w(I\setminus \{v_j\})+w(v_j)=w(I_j) + w(v_j) = w(I_j\cup\{v_j\})\]
and
\[w(I')=w(I'\setminus \{v_k\})+w(v_k)=w(I_{k}) + w(v_{k})= w(I_k\cup\{v_k\})\,.\]
The above two expressions are equal by assumption, so we get $w(I)=w(I')$ and thus $G$ is $w$-well-covered.
\end{proof}

\begin{corollary}
\label{cor:anti-neighborhood}
Let $G$ be a graph with vertex set $\{v_1,\dots, v_n\}$.
For each $j\in [n]$ let $\W_j$ be a rational (resp.~integer or unit) well-covering system of $G-N[v_j]$.
Then a rational (resp.~integer or unit) well-covering system of $G$ with size $\sum_{j = 1}^n|\W_j|+n-1$ can be computed in time $\mathcal{O}(n(n+m)+\sum_{j = 1}^n|\W_j|)$, where $m = |E(G)|$.
\end{corollary}

\begin{proof}
In time $\mathcal{O}(n(n+m))$ we compute the graphs $G-N[v_j]$ for all $j\in [n]$ and a maximal independent set $I_j$ in each such graph.
Then, using Lemma~\ref{lem:equations-anti-neighborhood} we compute a well-covering system of $G$ in time $\mathcal{O}(\sum_{j=1}^n |\W_j|+n^2)$.
The total complexity of this approach is $\mathcal{O}(n(n+m)+\sum_{j = 1}^n|\W_j|)$, as claimed.
\end{proof}

Using the above result, we give a more general statement, which will be an ingredient of the main algorithm in this paper.

\begin{theorem}\label{thm:reduction-to-anti-neighborhoods}
Let $\G$ and $\G^*$ be two graph classes such that for every graph $G$ in $\G$ and every vertex $v$ of $G$ the graph $G-N[v]$ is in $\G^*$.
Assume that for each graph $G$ in $\G^*$ with $n$ vertices and $m$ edges one can compute in time $f(n,m)$ a rational (resp.~integer or unit) well-covering system of $G$ with size at most $g(n,m)$, where $f$ and $g$ are nondecreasing functions.
Then for any graph $G$ in $\mathcal{G}$ with $n$ vertices and $m$ edges, one can compute in time $\mathcal{O}(n\cdot(n+m + f(n,m)))$ a rational (resp.~integer or unit) well-covering system of $G$ with size at most $n\cdot g(n,m)+n-1$.
\end{theorem}

\begin{proof}
Let $G$ be a graph in $\G$ with vertex set $V(G)=\{v_1,\ldots, v_n\}$ and let $m = |E(G)|$.
For all $j\in [n]$, let $G_j=G-N[v_j]$.
The graphs $G_j$, $j\in [n]$, can be computed in time $\mathcal{O}(n(n+m))$.
By assumption, for each $j\in[n]$ the graph $G_j$ is in $\G^*$, and hence a rational (resp.~integer or unit) well-covering system $\W_j$ of $G_j$ with at most $g(|V(G_j)|,|E(G_j)|)\le g(n,m)$ equations can be computed in time $f(|V(G_j)|, |E(G_j)|) \le f(n,m)$.
Note also that $|\W_j|\le f(|V(G_j)|, |E(G_j)|) \le f(n,m)$.
By~\Cref{cor:anti-neighborhood}, a well-covering system of $G$ with size $\sum_{j = 1}^n|\W_j|+n-1\le n\cdot g(n,m)+n-1$ can be computed in time $\mathcal{O}(n(n+m)+\sum_{j = 1}^n|\W_j|) = \mathcal{O}(n\cdot (n+m+ f(n,m))$.
\end{proof}

\section{Fork-free graphs}\label{sec:fork-free}

By \Cref{thm:cographs}, a well-covering system of a given cograph can be computed in polynomial time.
In this section, we generalize the result of~\Cref{thm:cographs} to prove the main result of this paper, a polynomial-time algorithm for computing a well-covering system of a given fork-free graph.
This is a significant generalization of \Cref{thm:cographs}, since the class of fork-free graphs also generalizes the class of claw-free graphs.
Our approach combines the results from~\Cref{sec:prime,sec:anti-neighborhoods} with a known structural result on fork-free graphs, which allows us to reduce the problem to the class of claw-free graphs, for which the following theorem applies.

\begin{theorem}[Levit and Tankus~\cite{levit2015weighted}]
\label{thm:claw-free}
There exists an $\mathcal{O}(n^3m^{3/2})$ algorithm, which receives as input a claw-free graph $G$ with $n$ vertices and $m\ge 1$ edges and computes a unit well-covering system of $G$.
\end{theorem}

Following~\Cref{remark:modular} and the fact that the function $f$ defined by the rule $f(n,m) = n^{\omega+2}m^{3/2}$ for all $m,n\ge 0$, is superadditive, \Cref{thm:claw-free} has the following consequence.

\begin{corollary}\label{cor:prime-are-claw-free}
Let $\C$ be the class of all graphs $G$ such that every prime induced subgraph of $G$ is claw-free.
Then for any graph $G$ in $\C$ with $n$ vertices and $m\ge 1$ edges, one can compute in time $\mathcal{O}\big(n^{\omega+2}m^{3/2}\big)$ a unit well-covering system of $G$ with size at most $n$.
\end{corollary}

To apply~\Cref{cor:prime-are-claw-free}, we use the following structural result on fork-free graphs due to Lozin and Milani\v{c}~\cite{MR2373852,lozin2008polynomial}.\footnote{The result is stated incorrectly in the paper~\cite{lozin2008polynomial}.
It is stated correctly in the conference version of that work~\cite{MR2373852}, as well as in the PhD thesis~\cite[Theorem 3.1.2]{milanic2007algorithmic}, and it is reproved by Dyer, Martin, Jerrum, M\"{u}ller, and Vu{\v{s}}kovi\'{c}~in~\cite{MR4276974}.}

\begin{theorem}
\label{thm:fork-free}
Let $G$ be a prime fork-free graph, let $x$ be a vertex of $G$, and let $G'$ be a prime induced subgraph of the graph $G-N[x]$.
Then $G'$ is claw-free.
\end{theorem}

Using~\cref{thm:reduction-to-anti-neighborhoods,cor:prime-are-claw-free,thm:fork-free}, we can now derive the following.

\begin{lemma}\label{lem:prime-fork-free}
Given a prime fork-free graph $G$ with $n$ vertices and $m\ge 1$ edges, a unit well-covering system of $G$ with size at most $n$ can be computed in time $\mathcal{O}(n^{\omega+3}m^{3/2})$.
\end{lemma}

\begin{proof}
Let $\mathcal{F}$ be the class of prime fork-free graphs and let
$\mathcal{F}^*$ be the class of all graphs $G$ such that every prime induced subgraph of $G$ is claw-free.
By~\cref{thm:fork-free}, for every graph $G\in \mathcal{F}$ and every vertex $x\in V(G)$, the graph $G-N[x]$ belongs to $\mathcal{F}^*$.
By~\cref{cor:prime-are-claw-free}, given a graph $G\in \mathcal{F}^*$ with $n$ vertices and $m$ edges one can compute in time \hbox{$\mathcal{O}\big(n+n^{\omega+2}m^{3/2}\big)$} a unit well-covering system of $G$ with size at most $n$, where the additive $\mathcal{O}(n)$ term has only been added in order to allow for $G$ to be edgeless.
Thus, by \cref{thm:reduction-to-anti-neighborhoods}, given a graph $G\in \mathcal{F}$ with $n$ vertices and $m\ge 1$ edges one can compute in time \hbox{$\mathcal{O}\big(n\cdot \big(n+m+n^{\omega+2}m^{3/2}\big)\big) = \mathcal{O}\big(n^{\omega+3}m^{3/2}\big)$} a unit well-covering system $\widehat{\W}$ of $G$ with size at most $n^2+n-1$.
By \cref{lem:reduced-well-covering}, a unit well-covering subsystem $\W\subseteq \widehat{\W}$ of $G$, with size at most $n$, can be computed in time $\mathcal{O}(n^{\omega-1}|\widehat{\W}|)=\mathcal{O}(n^{\omega+1})$.
The total time complexity of this approach is $\mathcal{O}\big(n^{\omega+3}m^{3/2}\big) +\mathcal{O}\big(n^{\omega+1}\big) = \mathcal{O}\big(n^{\omega+3}m^{3/2}\big)$, as claimed.
\end{proof}

We now have everything ready to prove the main result of the paper.

\begin{theorem}\label{thm:main}
Given a fork-free graph $G$ with $n$ vertices and $m\ge 1$ edges, a unit well-covering system of $G$ with size at most $n$ can be computed in time $\mathcal{O}(n^{\omega+3}m^{3/2})$.
\end{theorem}

\begin{proof}
Let $\G$ be the class of fork-free graphs and $\G^*$ the class of prime fork-free graphs.
\cref{lem:prime-fork-free} implies that given a graph $G$ in $\G^*$ with $n$ vertices and $m\ge 1$ edges, a unit well-covering system of $G$ with size at most $n$ can be computed in time $\mathcal{O}(n^{\omega+3}m^{3/2})$.
Let $f(n,m)=n^{\omega+3}m^{3/2}.$
By \cref{thm:modular+Gauss}, given a fork-free graph $G$ with $n$ vertices and $m\ge 1$ edges, a unit well-covering system $\W$ of $G$ with size at most $n$ can be computed in time $\mathcal{O}(f(2n,m)+n^{\omega+1})=\mathcal{O}((2n)^{\omega+3}m^{3/2}+n^{\omega+1})$, which simplifies to $\mathcal{O}(n^{\omega+3}m^{3/2})$.
\end{proof}

We can determine if a graph $G$ is well-covered by computing a well-covering system of $G$ and checking if the weight function assigning $1$ to each vertex of $G$ satisfies all the equations in the system.
This leads to the following consequence of~\Cref{thm:main}.

\begin{corollary}
There is a polynomial-time algorithm to determine if a given fork-free graph is well-covered.
\end{corollary}

\section{Concluding remarks}\label{sec:conclusion}

In this paper we developed two general reductions for the problem of computing a well-covering system of a given graph, that is, a system of linear homogeneous equations representing the well-covered vector space of the graph.
Using these reductions, we showed that the problem can be solved in polynomial time in the class of fork-free graphs.
For the special case of cographs, a faster algorithm was developed.

As a promising avenue for future research, it would be interesting to study the problem in further generalizations of the class of cographs, for example, in the classes considered \hbox{in~\cite{MR3851408,Araujo2019}}, including classes of bounded cliquewidth, in which the well-coveredness property can be recognized in \FPT time (with cliquewidth as the parameter, see~\cite{MR3851408}).
The complexity of computing the well-covered dimension of a graph, as well as the special case of recognizing graphs with positive well-covered dimension also seem to
be questions worthy of further consideration.

\paragraph{Statements and Declarations.}
\textit{Declaration of interest:} The authors have no conflicts of interest to declare that are relevant to the content of this article.

\paragraph{Acknowledgements.}
We are grateful to the anonymous reviewer for helpful remarks.
This work is supported in part by the Slovenian Research and Innovation Agency (I0-0035, research program P1-0285, research projects J1-3001, J1-3002, J1-3003, J1-4008, J1-4084, N1-0102, and N1-0160, and a Young Researchers Grant) and by the research program CogniCom (0013103) at the University of Primorska.

\end{document}